\documentclass[a4paper,oneside]{amsart}

\usepackage[textwidth=125mm,textheight=195mm]{geometry}
\usepackage{amssymb,amsthm,amsmath}
\usepackage{mathtools}
\usepackage{microtype}
\usepackage{graphicx}
\usepackage{float}
\usepackage{enumitem}
\usepackage{mathrsfs}
\usepackage{hyperref}

\usepackage[all]{xy}

\makeatletter
\@namedef{subjclassname@2020}{%
  \textup{2020} Mathematics Subject Classification}
\makeatother

\theoremstyle{plain}
\newtheorem{thm}{Theorem}[section]
\newtheorem*{thmnn}{Theorem}
\newtheorem{prop}[thm]{Proposition}
\theoremstyle{definition}
\newtheorem{defi}[thm]{Definition}

\newcommand{\mc}[1]{\mathcal #1}
\newcommand{\TO}{\mc L}

\newcommand{\eps}{\varepsilon}
\newcommand{\wh}{\widehat}

\newcommand\N{\mathbb{N}}

\newcommand\R{\mathbb{R}}

\newcommand\C{\mathbb{C}}

\newcommand{\h}{\mathbb{H}}

\newcommand{\bmat}[4]{\begin{bmatrix} #1&#2\\#3&#4\end{bmatrix}}
\newcommand{\textmat}[4]{\left(\begin{smallmatrix} #1&#2 \\ #3&#4
\end{smallmatrix}\right)}
\newcommand{\textbmat}[4]{\left[\begin{smallmatrix} #1&#2 \\ #3&#4
\end{smallmatrix}\right]}

\DeclareMathOperator{\id}{id}
\DeclareMathOperator{\SL}{SL}
\DeclareMathOperator{\PSL}{PSL}
\DeclareMathOperator{\PSO}{PSO}
\DeclareMathOperator{\SO}{SO}
\DeclareMathOperator{\Ima}{Im}

\DeclareMathOperator{\Stab}{Stab}
\DeclareMathOperator{\WC}{C\!}

\DeclareMathOperator{\ad}{ad}
\DeclareMathOperator{\pr}{pr}

\newcommand{\globspace}{\mathbb{X}}
\newcommand{\locspace}{\mathbb{Y}}

\newcommand{\mf}{\mathfrak}
\newcommand{\reg}{\textnormal{reg}}
\newcommand{\st}{\mathnormal{st}}

\begin{document}
\title[Symbolic dynamics and transfer operators]{Symbolic dynamics and transfer 
operators for Weyl chamber flows: a class of examples}

\author{Anke Pohl} 
\address{Anke Pohl, University of Bremen, Department~3 -- Mathematics, 
Bibliothekstr.~5, 28359 Bremen, Germany}
\email{apohl@uni-bremen.de}

\begin{abstract}
We provide special cross sections for the Weyl chamber flow on a sample class 
of Riemannian locally symmetric spaces of higher rank, namely the direct 
product spaces of Schottky surfaces. We further present multi-parameter 
transfer operator families for the discrete dynamical systems on Furstenberg 
boundary that are related to these cross sections.
\end{abstract}

\subjclass[2020]{Primary: 37B10, 37C30; Secondary: 22E40, 53C35}
\keywords{Weyl chamber flow, cross section, discrete dynamical system, transfer 
operator}

\maketitle

\section{Introduction}

Discretizations of flows on various types of spaces and symbolic dynamics for 
them are useful for many purposes. The applications that motivate this article 
are transfer operator approaches to Laplace eigenfunctions, resonances and 
dynamical zeta functions as, e.g., in \cite{Ruelle_zeta, 
Fried_zeta, Mayer_thermo, Mayer_thermoPSL, Pollicott, 
Morita_transfer, Chang_Mayer_transop, Patterson_Perry, 
Mayer_Muehlenbruch_Stroemberg, Moeller_Pohl, Pohl_mcf_Gamma0p, Pohl_mcf_general, 
Adam_Pohl_iso_hecke, Fedosova_Pohl, Pohl_Zagier, Bruggeman_Pohl}. At the current 
state of art, these are limited to hyperbolic spaces, and mostly even to 
hyperbolic surfaces. 
The necessary discretizations of the geodesic flow on these spaces are 
typically constructed by means of cross sections. For the dynamical approaches 
to Laplace eigenfunctions referred to above it was discovered to be crucial to 
deviate from the classical notion of cross sections. For these applications a 
relaxed notion is used where we require from a cross section that it detects 
all \emph{periodic} geodesics, but it need not have ``time unbounded'' 
intersections with \emph{all} geodesics, in particular not with those that 
eventually stay in an end of the considered space. In contrast, the classical 
notion would require that a cross section has infinitely many intersections with 
\emph{all} geodesics in both of their ``time directions'', past and future. Also 
transfer operator approaches to dynamical zeta functions such as Ruelle and 
Selberg zeta functions were seen to benefit from this more flexible notion of 
cross section.  

In this article we discuss a notion of cross section for the Weyl chamber flow 
on Riemannian locally symmetric spaces that is modelled in analogy to this 
relaxed notion of cross sections for one rank spaces. We construct such 
cross sections for a sample class of spaces. We further provide related 
discrete dynamical systems on the Furstenberg boundary of the Riemannian 
globally symmetric spaces covering these spaces and we present associated 
multi-parameter transfer operator families. This work may be seen as a first 
step of an attempt towards transfer operator approaches for higher rank spaces. 
However, cross sections for Weyl chamber flows are sought for also for other 
purposes. See, e.g.,~\cite[Question~28]{Gorodnik_problems}. Our considerations 
here might also be of use for such questions.

In a nutshell, we define a cross section for the Weyl chamber flow to be a 
subset of the Weyl chamber bundle (or the set of Weyl chambers) that is 
intersected by all \emph{compact} oriented flats and for which each 
intersection with any oriented flat is discrete in the \emph{regular} 
directions of the flow. The first requirement is motivated by the fact that 
compact oriented flats are the Weyl-chamber-flow equivalent of periodic  
geodesics for the geodesic flow. For spaces of rank one, the compact oriented 
flats are precisely the periodic geodesics. The second requirement is motivated 
by the idea that intersections should be only momentarily. The restriction of 
this requirement to regular directions results from the fact that in singular 
directions there are some space dimensions of the considered flat in which no 
motion happens. We allow non-discreteness in the space directions without 
motion. 

The sample spaces of higher rank that we consider here are the direct products 
of Schottky surfaces. These spaces have a rather simple structure, but 
nevertheless their study in regard to our goals is very instructive. For 
Schottky surfaces, cross sections for the geodesic flow, induced discrete 
dynamical systems as well as one-parameter transfer operator families are 
well-known. We will take advantage of this knowledge. As cross section for 
the geodesic flow on a Schottky surface one typically takes the set of unit 
tangent vectors that are based on the boundary of a standard fundamental 
domain for the considered surface and that are directed to the interior of the 
fundamental domain. We call such a cross section \emph{standard} for 
the moment and refer to Section~\ref{sec:schottky} for details and precise 
definitions. Our main results regarding these spaces are essentially as follows.

\begin{thmnn}[Coarse statement]
Let $r\in\N$. For $j\in\{1,\ldots, r\}$ let $\locspace_j$ be a Schottky surface 
and $\wh C_j$ a standard cross section for the geodesic flow on~$\locspace_j$. 
Let 
\[
 \locspace \coloneqq \locspace_1\times \cdots \times \locspace_r
\]
be the Riemannian locally symmetric space of rank~$r$ which is the direct 
product of the Schottky surfaces~$\locspace_j$, $j\in\{1,\ldots, r\}$. 
\begin{enumerate}[label=$\mathrm{(\roman*)}$, ref=$\mathrm{\roman*}$]
 \item The direct product
 \[
  \wh C \coloneqq \wh C_1 \times \cdots \times \wh C_r
 \]
provides a cross section for the Weyl chamber flow on~$\locspace$. 
\item The first return map of~$\wh C$ is \mbox{(semi-)}conjugate to a discrete 
dynamical system~$F$ on the Furstenberg boundary of the Riemannian globally 
symmetric space~$\globspace$ covering~$\locspace$. The map~$F$ is piecewise 
given by the action of certain elements of the fundamental group of~$\locspace$.
\item The multi-parameter transfer operator family associated to~$F$ can be 
provided in explicit form.
\end{enumerate}
\end{thmnn}

The refined version of this theorem features explicit expressions and 
formulas at each level. It appears as Theorems~\ref{thm:cs_wcf}, 
\ref{thm:stcross_wcf} and~\ref{thm:inducedmap} and as Section~\ref{sec:TO}. A 
few remarks are in order.

\begin{itemize}
\item The choice of~$\wh C$ as direct product of the cross sections~$\wh C_j$ 
for the Schottky surfaces may seen trivial at first glimpse. However, it shows 
important properties of cross sections for Weyl chamber flows of spaces of 
higher rank. For each $j\in\{1,\ldots, r\}$, the base point set of~$\wh C_j$ is 
the full boundary of the fundamental domain for~$\locspace_j$. The base point 
set of~$\wh C$ is only a small part of the boundary of a fundamental 
domain for~$\locspace$. Thus, the cross section~$\wh C$ is a rather sparse 
subset of the Weyl chamber bundle of~$\locspace$, and it does not arise from a 
Koebe--Morse method (as the~$\wh C_j$ do). 
\item The first return map relies on the notion of ``next'' intersections. 
Since the flow is multi-dimensional as soon as the rank $r$ is larger than~$1$, 
the concept of next intersections needs to be discussed with care. It is 
introduced in Section~\ref{sec:def_cross}.
\item The multi-parameter transfer operator family associated to the map~$F$ 
is \emph{not} the direct product of the one-parameter transfer operator 
families for~$\locspace_j$, $j\in\{1,\ldots, r\}$. This family is discussed 
in detail in Section~\ref{sec:TO}.
\end{itemize}

This article is structured as follows. In Section~\ref{sec:concepts} we first 
survey the necessary background on Riemannian globally and locally symmetric 
spaces. In Section~\ref{sec:def_cross} we then introduce the notions of cross 
sections, next intersections, first return map and induced discrete dynamical 
systems on Furstenberg boundary. Section~\ref{sec:schottky}--\ref{sec:TO} are 
devoted to the study of our sample class of spaces. In 
Section~\ref{sec:schottky} we briefly present the well-known cross sections and 
related objects for Schottky surfaces. In Section~\ref{sec:cs} we provide and 
discuss a cross section for the Weyl chamber flow. In Section~\ref{sec:coding} 
we introduce the discrete dynamical system on Furstenberg boundary that is 
induced by a well-chosen set of representatives for the cross section. In the 
final Section~\ref{sec:TO} we present the multi-parameter transfer operator 
family associated to this discrete dynamical system.

\subsection*{Acknowledgement} 
This research was funded by the Deutsche Forschungsgemeinschaft (DFG, German 
Research Foundation) -- project no.~264148330.

\section{Geometry and dynamics of Riemannian locally symmetric 
spaces}\label{sec:concepts}

In this section we introduce the elements from the geometry and dynamics of 
Riemannian globally and locally symmetric spaces that we will use further on. 
Since there are many excellent treatises of this subject, we will keep our 
presentation to the bare minimum necessary for our purposes and will omit any 
proofs. For comprehensive expositions, we refer in particular 
to~\cite{ImHof, Eberlein, Helgason_diffgeo}. We 
further present the notion of cross section for the Weyl chamber flow, 
discuss the concept of next intersections and the existence of the first return 
map. We also provide the notion of a set of representatives for the cross 
section and the notion of an associated discrete dynamical system on 
Furstenberg boundary.

\subsection{Globally symmetric spaces}\label{sec:globspace}

Let $\globspace$ be a Riemannian symmetric space of noncompact type. A 
\emph{flat} of~$\globspace$ is any totally geodesic, flat submanifold 
of~$\globspace$ of maximal dimension among all such submanifolds. The common 
dimension of all flats is the \emph{rang} of~$\globspace$. Let $F$ be a flat 
of~$\globspace$, pick a point $x\in F$ and consider the union of all flats 
that are distinct from~$F$ but contain the point~$x$. This union intersects~$F$ 
in a finite union of hyperplanes of~$F$. The connected components of the (open) 
complement of this union of hyperplanes in~$F$ are the \emph{Weyl chambers} 
in~$F$ with \emph{base point}~$x$. We let 
\[
 \WC\globspace \coloneqq \{ \text{Weyl chamber in~$F$ with base point~$x$} 
\mid \text{$F$ flat, $x\in F$} \}
\]
denote the set of all Weyl chambers of~$\globspace$, i.e., the family of all 
Weyl chambers of all flats with all base points. We note that each Weyl 
chamber~$c\in\WC\globspace$ is contained in a unique flat~$F$ and has a unique 
base point~$x$.  

For the introduction of further objects, in particular the Weyl chamber flow, 
we take advantage of some elements of the structure theory of~$\globspace$. We 
fix an arbitrary point~$x_0$ 
in~$\globspace$, an \emph{origin} or \emph{reference point}, let $G$ be 
the identity component of the group of Riemannian isometries of~$\globspace$, 
and let $K\coloneqq \Stab_{G}(x_0)$ denote the stabilizer group of~$x_0$ 
in~$G$. Then $G$ is a real semisimple Lie group with finite center, 
and $K$ is a maximal compact subgroup of~$G$. We may identify the 
symmetric space~$\globspace$ with the homogeneous space $G/K$ via the 
isomorphism 
\begin{equation}\label{eq:isospace}
 G/K \to \globspace\,,\quad gK \mapsto g(x_0)\,.
\end{equation}
We fix a flat~$F_0$ through~$x_0$, a \emph{reference flat}. Then there exists a 
unique maximal abelian subgroup~$A$ of~$G$ such that $F_0 = A(x_0)$. 
Indeed, $F_0$ may be identified with~$A$ via the isomorphism
\begin{equation}\label{eq:isoflat}
 A \to F_0\,,\quad a \mapsto a(x_0)\,.
\end{equation}
We now fix a Weyl chamber~$c_0$ in~$F_0$ with base point~$x_0$, a 
\emph{reference Weyl chamber}, and let $M \coloneqq Z_K(A)$ denote the 
centralizer of~$A$ in~$K$. Then we may interpret the set~$\WC\globspace$ as 
the homogeneous space~$G/M$ via the isomorphism
\begin{equation}\label{eq:isoWC}
 G/M \to \WC\globspace\,,\quad gM \mapsto g(c_0)\,.
\end{equation}
Let 
\begin{equation}\label{eq:projbase}
\pi_B\colon \WC\globspace\to\globspace\,,\quad c\mapsto x_c\,,
\end{equation}
denote the map from a Weyl chamber~$c$ to its base point~$x_c$. Under the 
isomorphisms from~\eqref{eq:isospace} and~\eqref{eq:isoWC}, the map~$\pi_B$ 
becomes the quotient map 
\begin{equation}\label{eq:quotbase}
 G/M \to G/K\,,\quad gM \mapsto gK\,,
\end{equation}
which also turns $\WC\globspace$ into a bundle over~$\globspace$. We further 
let $M'\coloneqq N_K(A)$ denote the normalizer of~$A$ in~$K$, and let 
$W\coloneqq M'/M$ denote the \emph{Weyl group}. With respect to the isomorphism 
in~\eqref{eq:isoWC}, the action of~$G$ on~$\WC\globspace$ becomes 
\begin{equation}\label{eq:GonWC}
 G\times G/M \to G/M\,,\quad (h,gM) \mapsto hgM\,,
\end{equation}
and the action of the Weyl group~$W$ on~$\WC\globspace$ becomes
\begin{equation}\label{eq:WonWC}
 W\times G/M \to G/M\,,\quad (m'M, gM) \mapsto gm'M\,.
\end{equation}
We immediately see that the set of all Weyl chambers in~$F_0$ with base 
point~$x_0$ are the (finitely many and pairwise distinct) elements of the orbit 
of the reference Weyl chamber~$c_0$ under the Weyl group~$W$. Further, the set 
of all flats of~$\globspace$ consists of all $G$-translates of the 
reference flat~$F_0$. The \emph{Weyl chamber flow} on~$\globspace$ is the action 
of~$A$ on~$\WC\globspace$ which under the isomorphism in~\eqref{eq:isoWC} 
becomes 
\begin{equation}\label{eq:WCF}
A\times G/M \to G/M\,,\quad (a,gM)\mapsto gaM\,. 
\end{equation}
The isomorphism in~\eqref{eq:isoflat} implies that the Lie group~$A$ 
is isomorphic to~$\R^r$, where $r$ is the rang of~$\globspace$. Thus, the 
Weyl chamber flow is a flow on~$\WC\globspace$ with $r$ ``time'' dimensions. 
The (unoriented) flow ``lines'' of the Weyl chamber flow on~$\globspace$ are 
the flats. More precisely, for each each Weyl chamber~$gM$, the base point set 
of its $A$-orbit is 
\[
 \pi_B(gAM) = gAK\,,
\]
which is the unique flat that contains~$gM$. For Riemannian symmetric spaces of 
rang one, the Weyl chamber flow coincides with the geodesic flow. 

In what follows we introduce a rather coarse notion of orientation of flats 
(following~\cite{ImHof}), which provides an appropriate notion of 
directions for the Weyl chamber flow and will be crucial for our notion of 
cross sections. We further introduce the Furstenberg boundary, which provides 
the appropriate geometry at infinity for our codings in 
Section~\ref{sec:coding}. To that end we note that the choice of the reference 
Weyl chamber~$c_0$ distinguishes an open subset, $A^+$, of~$A$ by means of the 
isomorphism in~\eqref{eq:isoflat}. The subset~$A^+$ is often called the open
\emph{positive Weyl chamber} in~$A$, subject to the choice of~$c_0$.
We let $\mf g$ and~$\mf a$ denote the Lie algebra of~$G$ and~$A$, respectively, 
and let $\Lambda$ be the set of roots of $(\mf g,\mf a)$. We let $\mf a^+$ 
denote the subset of~$\mf a$ that corresponds to~$A^+$ under the exponential 
map. Then 
\[
 \Lambda^+ \coloneqq \{ \lambda \in \Lambda \mid \forall\, H\in\mf 
a^+\colon \lambda(H) > 0\}
\]
is a choice of positive roots of~$(\mf g,\mf a)$. We set 
\[
 \mf n \coloneqq \bigoplus_{\lambda\in\Lambda^+}\mf g_\lambda\,,
\]
where 
\[
 \mf g_\lambda \coloneqq \{ X\in\mf g \mid \forall\, H\in\mf a\colon 
\ad_H(X) = \lambda(H)X \}
\]
is the root space of~$\lambda$, and $\ad$ is the adjoint representation of~$\mf 
g$. We let $N \coloneqq \exp\mf n$ be the (unipotent) subgroup associated 
to~$\mf n$. We note that $N$ completes the pair $(K,A)$ to an Iwasawa 
decomposition of~$G$. Thus, the map 
\[
 N\times A \times K \to G\,,\quad (n,a,k) \mapsto nak\,,
\]
is an isomorphism of Lie groups. We let $P \coloneqq  NAM$ be the associated 
minimal parabolic subgroup of~$G$. The \emph{Furstenberg boundary} 
of~$\globspace$ is the homogeneous space 
\[
 G/P \,.
\]
We let 
\begin{equation}\label{eq:projFB}
 \nu\colon G/M \to G/P\,,\quad gM \mapsto gP\,,
\end{equation}
denote the canonical projection from the Weyl chamber bundle to the Furstenberg 
boundary. In order to define a notion of orientation of flats, we call two Weyl 
chambers $gM$ and $hM$ \emph{asymptotic} if they project to the same point in 
the Furstenberg boundary, thus if
\[
 \nu(gM) = \nu(hM)\,.
\]
This property induces an equivalence relation on the set of all Weyl 
chambers. The combined map 
\begin{equation}\label{eq:combmap}
 \alpha\coloneqq (\pi_B,\nu)\colon G/M \to G/K \times G/P
\end{equation}
is an isomorphism. Therefore, the Furstenberg boundary~$G/P$ can be interpreted 
as the set of equivalence classes of asymptotic Weyl chambers. 

An \emph{orientation} or \emph{direction} of a flat~$F$ is an 
equivalence class of asymptotic Weyl chambers that has a representative in~$F$. 
An \emph{oriented flat} is a flat endowed with a distinguished orientation. The 
set of oriented flats can be identified with the homogeneous space~$G/AM$ 
via the isomorphism
\begin{equation}\label{eq:isoflatorient}
 G/AM \to \{\text{oriented flats}\}\,,\quad gAM \mapsto \bigl(g(F_0), 
gP\bigr)\,.
\end{equation}
The map $\nu$ in~\eqref{eq:projFB} splits into the two canonical maps 
\begin{equation}\label{eq:nu12}
 \nu_1\colon G/M\to G/AM\qquad\text{and}\qquad \nu_2\colon 
G/AM\to G/P\,,
\end{equation}
where $\nu_1$ maps a Weyl chamber~$gM$ to the flat that contains it and 
endows this flat with the equivalence class of~$gM$ as orientation, and $\nu_2$ 
projects an oriented flat to its orientation, identified with the point 
in the Furstenberg boundary~$G/P$. The actions of~$G$ and~$W$ on~$G/M$ 
from~\eqref{eq:GonWC} and~\eqref{eq:WonWC} descend to actions on~$G/AM$ 
and~$G/P$, turning $\nu_1$ and~$\nu_2$ into $G$-equivariant as well as 
$W$-equivariant maps. In particular, we have
\[
 W\times G/AM \to G/MA\,,\quad (m'M, gMA) \mapsto gm'MA\,,
\]
and
\[
 W\times G/P\to G/P\,,\quad (m'M, gP) \mapsto gm'P\,.
\]
For any flat $gF_0$, the possible orientations are therefore characterized by 
the finitely many points $gm'P$, where $m'$ runs through a representative set 
of~$W$ in~$M'$.

\subsection{Locally symmetric spaces}

We continue to use the notation from the previous section and now 
apply the identifications discussed there without mentioning the 
isomorphisms. In particular, we allow ourselves to write $\globspace = G/K$, 
and analogously for other objects. We let $\Gamma$ be a discrete subgroup 
of~$G$. The quotient space
\[
 \locspace \coloneqq \Gamma\backslash\globspace = \Gamma\backslash G/K
\]
is a locally symmetric space or more precisely, if $\Gamma$ has torsion, an 
orbifold. We let 
\begin{equation}\label{eq:projY}
 \pi^\Gamma \colon \globspace \to \locspace\,,\quad gK \mapsto \Gamma gK\,,
\end{equation}
denote the canonical quotient map. Most of the objects defined 
for~$\globspace$ in the previous section descend to analogous 
objects for~$\locspace$, via~$\pi^\Gamma$. The \emph{rank}, $r$,  
of~$\locspace$ is the rank of~$\globspace$. The 
$\pi^\Gamma$-images of the flats of~$\globspace$ are called the \emph{flats} 
of~$\locspace$. We remark that flats of~$\locspace$ are not necessarily 
isometric to~$\R^r$, much in contrast to flats of~$\globspace$. In particular, a 
flat of~$\locspace$ might be compact (as a subset of~$\locspace$). For locally 
symmetric spaces of rang one, compact flats coincide with periodic geodesics, 
more precisely, with the subsets of~$\locspace$ traced out by periodic 
geodesics. 

The set of \emph{Weyl chambers} of~$\locspace$ is 
\begin{equation}\label{eq:WConY}
 \WC\locspace = \Gamma\backslash G/M\,,
\end{equation}
the \emph{Weyl chamber flow} on~$\locspace$ is 
\begin{equation}\label{eq:WCFonY}
 A\times \Gamma\backslash G/M\to \Gamma\backslash G/M\,,\quad (a,\Gamma 
gM)\mapsto \Gamma gaM\,,
\end{equation}
and the set of \emph{oriented flats} of~$\locspace$ is 
\begin{equation}\label{eq:FlatonY}
 \Gamma\backslash G/AM\,.
\end{equation}
Since the projection maps~$\pi_B$ and~$\nu_1$ in~\eqref{eq:projbase} 
and~\eqref{eq:nu12} are $\Gamma$-equivariant, they induce the analogous maps 
on~$\locspace$:
\begin{align*}
\pi_B^\Gamma &\colon \WC\locspace \to \locspace\,,\quad \Gamma gM\mapsto \Gamma 
gK\,,
\intertext{and}
\nu_1^\Gamma &\colon \Gamma\backslash G/M \to \Gamma\backslash G/AM\,,\quad 
\Gamma gM \mapsto \Gamma gAM\,.
\end{align*}

\subsection{Cross sections and induced discrete dynamical 
systems}\label{sec:def_cross}

We resume the notation from the previous two sections. We let 
\begin{equation}\label{eq:Areg}
 A^\reg \coloneqq W(A^+)\,,
\end{equation}
denote the set of \emph{regular elements} in~$A$. For any subset~$S\subseteq 
G/M$, we say that an oriented flat $gAM$ \emph{intersects}~$S$ in~$hM$ if 
\[
 \nu_1(hM) = gAM\,.
\]
We say that the intersection is \emph{discrete} if there exists a 
neighborhood~$U$ of the identity element~$\id$ in~$A$ such that for all $a\in 
U\cap A^\reg$ we have
\[
 haM \notin S\,.
\]
We note that in this case, the Weyl chamber~$hM$ is contained in the oriented 
flat~$gAM$ and determines its orientation. Analogously, for any subset~$\wh 
S \subseteq \Gamma\backslash G/M$ we say that an oriented flat~$\Gamma gAM$ 
of~$\locspace$ \emph{intersects}~$\wh S$ in the Weyl chamber~$\Gamma hM$ 
of~$\locspace$ if 
\[
 \nu_1^\Gamma(\Gamma hM) = \Gamma gAM\,,
\]
and we call the intersection \emph{discrete} if there exists a 
neighborhood~$U$ of~$\id$ in~$A$ such that for all $a\in U\cap 
A^\reg$ we have 
\[
 \Gamma haM\notin \wh S\,.
\]
With these preparations we can now propose the following notion of cross 
section. 

\begin{defi}\label{def:cs}
We call a subset~$\wh C$ of~$\Gamma\backslash G/M$ a \emph{cross 
section} for the Weyl chamber flow on~$\locspace$ if 
\begin{enumerate}[label=$\mathrm{({C\arabic*})}$, ref=$\mathrm{({C\arabic*})}$]
 \item\label{cs1} every \emph{compact} oriented flat of~$\locspace$ 
intersects~$\wh C$, and
 \item\label{cs2} each intersection of any flat of~$\locspace$ with~$\wh C$ is 
discrete.
\end{enumerate}
\end{defi}

We emphasize the following aspects of this definition.
\begin{itemize}
\item We do not request that \emph{every} oriented flat shall intersect~$\wh 
C$. Definition~\ref{def:cs} is motivated by presumed properties necessary for 
transfer-operator-based investigations of the spectral theory of~$\locspace$.  
For rank one spaces, our previous investigations showed that for such 
applications, we only needed to request that all \emph{periodic} geodesics 
intersect a cross section for the geodesic flow. By density properties of these 
geodesics and a certain smoothness of the cross sections, it automatically 
meant that all geodesics that returned infinitely often to the compact core of 
the considered space intersect the cross section. For several results, it was 
crucial that those geodesics that eventually stay in the ends of the space, 
do not need to intersect the cross section at all or, when travelling along 
these geodesics, eventually stop intersecting it. For this reason, also for 
higher rank spaces, we only request that at least all compact oriented flats 
are detected by the cross section. In our examples in the following sections, 
we will see that all flats that are ``returning'' intersect the cross sections 
constructed there, but that flats ``vanishing to infinity'' eventually will 
not intersect anymore.  
\item We require discreteness of intersections only in the regular directions of 
the Weyl chamber flow, thus, for the action of~$A^\reg$. The application 
of~$a\in A^\reg$ on a Weyl chamber~$\Gamma gM$ causes motion in each space 
dimension of the flat that contains~$\Gamma gM$. In stark contrast, for $a\in 
A\setminus A^\reg$, there is no motion in some space dimensions. We allow 
non-discrete intersections in these dimensions. 
\item Since a set~$\wh C$ as in Definition~\ref{def:cs} is not a cross section 
in the classical sense, one might want to call it a ``cross section for the 
returning parts of the Weyl chamber flow in the regular time directions.'' 
\end{itemize}

We now turn to the definition of a first return map for a cross section of the 
Weyl chamber flow, where we aim to preserve the idea that this map should be 
given as follows. We pick a Weyl chamber~$\Gamma gM$ in~$\wh C$ and consider 
the oriented flat~$F = \Gamma gAM$ that contains~$\Gamma gM$. We move 
along~$F$ in the direction given by~$gM$, starting at~$\Gamma gM$ and 
ask for the ``next'' intersection of $F$ with~$\wh C$, say $\Gamma g'M$. The 
first return map should map~$\Gamma gM$ to~$\Gamma g'M$. Moving in the 
direction~$\nu(gM)$ is the same as restricting the flow to the positive Weyl 
chamber~$A^+$ (that deduces with the choice of the reference Weyl 
chamber~$c_0$, see Section~\ref{sec:globspace}). However, if the rank~$r$ 
of~$\locspace$ is larger than~$1$, then $A^+$ has $r$ time parameters and hence 
there does not need to be a well-defined ``first'' next intersection. We 
overcome this issue with Definition~\ref{def:nextinter} below, for which we 
start with a brief preparation.

The positive Weyl chamber~$A^+$ can be parametrized by an open cone 
in~$(\R_{>0})^r$, that is, by a convex subset $\tau^+$ of $(\R_{>0})^r$ such 
that for each $t=(t_1,\ldots, t_r)\in\tau^+$ the whole open ray 
\[
\R_{>0}\cdot t = \{ (ct_1,\ldots,ct_r) \mid c>0\}  
\]
is contained in~$\tau^+$. We fix such a parametrization 
\begin{equation}\label{eq:paramAplus}
 \tau^+ \to A^+\,,\quad t\mapsto a_t\,.
\end{equation}

\begin{defi}\label{def:nextinter}
Let $\wh C\subseteq \Gamma\backslash G/M$ be a cross section for the Weyl 
chamber flow on~$\locspace$. 
\begin{enumerate}[label=$\mathrm{(\roman*)}$, ref=$\mathrm{\roman*}$]
\item We say that $\Gamma gM \in \wh C$ has a 
\emph{future intersection} with~$\wh C$ if 
\[
\Gamma gA^+M\cap \wh C \not=\emptyset\,.
\]
In this case, let 
\[
 T\coloneqq \{ t = (t_1,\ldots, t_r)\in \tau^+ \mid \Gamma ga_tM \in\wh C\} 
\]
be the set of \emph{time vectors} of the future intersections. For 
$j\in\{1,\ldots, r\}$, let
\[
 \pr_j \colon \R^r \to \R\,, (t_1,\ldots, t_r) \mapsto t_j\,,
\]
be the projection on the $j$-th component, and set
\begin{align*}
 T_j &\coloneqq \pr_j(T)
 \\
 & \ = \{ t_j \mid \exists\, t_1,\ldots, t_{j-1},t_{j+1},\ldots, t_r\colon 
(t_1,\ldots, t_r) \in T\}\,.
\end{align*}
We say that $\Gamma gM$ has a \emph{next intersection} with~$\wh C$ if for all 
$j\in\{1,\ldots, r\}$, 
\[
 t_{0,j}\coloneqq \min T_j
\]
exists, 
\[
 t_0 \coloneqq \bigl(t_{0,1},\ldots, t_{0,r}\bigr) \in\tau^+\,
\]
and
\[
 \Gamma g a_{t_0} M \in \wh C\,.
\]
In this case, we call $t_0$ the \emph{first return time vector} of~$\Gamma gM$.
\item Let $\wh C_1$ be the subset of~$\wh C$ for which the first return vector 
exists. The \emph{first return map} is the map
\[
 \wh R\colon \wh C_1 \to \wh C\,,\quad \wh v=\Gamma g M \mapsto \Gamma 
ga_{t(\wh v)} M\,,
\]
where $t(\wh v)$ is the first return time vector of~$\wh v$. 
\end{enumerate}
\end{defi}

For the applications that motivate this article we need to be able to 
\mbox{(semi-)}con\-jugate the first return map to a function on the 
Furstenberg boundary of~$\globspace$. In what follows we present the necessary 
structures. 

Let $\wh C\subseteq \Gamma\backslash G/M$ be a cross section for the 
Weyl chamber flow of~$\locspace$. To determine the subset of~$\wh C$ on which 
the first return map~$\wh R$ becomes a self-map, we define iteratively for 
$n\in\N$, $n\geq 2$, the sets
\[
 \wh C_2 \coloneqq R\bigl(\wh C_1\bigr) \cap \wh C_1\,,\quad \wh C_3\coloneqq 
R\bigl(\wh C_2\bigr)\cap \wh C_1\,,\quad \ldots\,,
\]
hence
\[ 
 \wh C_{n} \coloneqq R\bigl(\wh C_{n-1}\bigr) \cap \wh C_{1}\qquad\text{for 
$n\in\N$, $n\geq 2$}\,,
\]
where $\wh C_1$ was defined in Definition~\ref{def:nextinter}. Then 
\begin{equation}\label{eq:crossstrong}
 \wh C_\st \coloneqq \bigcap_{n\in\N} \wh C_n
\end{equation}
is the subset of~$\wh C$ that consists of all those Weyl chambers in~$\wh C$ 
that yield an infinite sequence of successive next intersections with~$\wh C$. 
For each element in~$\wh C_\st$, the next intersection is obviously also 
contained in~$\wh C_\st$. Thus, if $\wh C_\st$ is nonempty, then the first 
return map~$\wh R$ restricts to a self-map of~$\wh C_\st$:
\[
 \wh R\colon \wh C_\st\to \wh C_\st\,.
\]
The set~$\wh C_\st$ may constitute a cross section on its own, in which case we 
call it the \emph{strong cross section} contained in~$\wh C$. (The subscript 
$\st$ refers to ``strong'' and is used here for the same motivation as 
in~\cite{Pohl_Symdyn2d}.)

We say that $C\subseteq G/M$ is a \emph{set of representatives} for~$\wh C$ if 
the quotient map 
\begin{equation}\label{eq:projWC}
 \pi^\Gamma\colon G/M \to \Gamma\backslash G/M\,,\quad gM \mapsto \Gamma gM\,,
\end{equation}
restricts to a bijection between~$C$ and~$\wh C$. (We use $\pi^\Gamma$ to 
denote both the map in~\eqref{eq:projY} and the map in~\eqref{eq:projWC}. This 
double use is motivated by the joint property of these maps to project to 
$\Gamma$-equivalence classes. The context will always clarify which instance 
of~$\pi^\Gamma$ is used.) If $C$ is any set of representatives for~$\wh C$ and 
$C_1$ is the subset of~$C$ that corresponds to~$\wh C_1$, then the first return 
map~$\wh R$ induces a map $R\colon C_1\to C$ which makes the diagram 
\[
\xymatrix{
C_1 \ar[r]^{R} \ar[d]_{\pi^\Gamma} & C \ar[d]^{\pi^\Gamma}
\\
\wh C_1 \ar[r]^{\wh R} & \wh C
}
\]
commutative. If $C_\st$ denotes the subset of~$C$ that corresponds to~$\wh 
C_\st$, then $R$ restricts to a self-map of~$C_\st$.

We recall the map $\nu\colon G/M\to G/P$ from~\eqref{eq:projFB} 
that projects the Weyl chambers of~$\globspace$ to the points in the 
Furstenberg boundary that are identified with their equivalence class 
of asymptotic Weyl chambers (and hence, in a certain sense, the direction of 
the Weyl chamber). For a well-chosen pair~$(\wh C, C)$ 
one may find a (unique) map $F\colon \nu(C_1) \to \nu(C)$ such that the diagram 
\[
\xymatrix{
C_1 \ar[r]^R \ar[d]_{\nu} & C \ar[d]^{\nu}
\\
\nu(C_1) \ar[r]^{F} & \nu(C)
}
\]
commutes. In this case and if $C_\st\not=\emptyset$, 
$F$ restricts to a self-map of~$\nu(C_\st)$. If $\wh C_\st$ is a cross section, 
we call
\[
F\colon \nu(C_\st) \to \nu(C_\st)
\]
the \emph{discrete dynamical system on the Furstenberg boundary} induced by 
$(\wh C, C)$. Typically, the map~$F$ is piecewise given by the action of 
certain elements from~$\Gamma$ on subsets of~$G/P$. The orbits of Weyl 
chambers~$\Gamma gM$ under the first return map~$\wh R$ relate to orbits of~$F$ 
and hence to sequences of the acting elements from~$\Gamma$. These sequences 
are often called \emph{coding sequences} for the oriented flats or Weyl 
chambers, and the shift along coding sequences provides a \emph{symbolic 
dynamics} for the Weyl chamber flow. 

We remark that the set~$C$ completely determines the cross 
section~$\wh C$. Therefore, for constructions of cross section we may start by
finding a ``nice'' set~$C$ and define a cross section as~$\pi^\Gamma(C)$. We 
also note that the knowledge of~$C$ is sufficient to determine the discrete 
dynamical system on the Furstenberg boundary. Thus, also this is induced by~$C$ 
alone.

In Sections~\ref{sec:cs} and~\ref{sec:coding} we will present for a certain 
class of locally symmetric spaces cross sections for their Weyl chamber flows 
as well as codings and associated discrete dynamical systems

\section{Schottky surfaces}\label{sec:schottky}

The locally symmetric spaces for which we will demonstrate the existence of 
cross sections for the Weyl chamber flow and induced discrete dynamical systems 
in the sense of Section~\ref{sec:def_cross} are product spaces of Schottky 
surfaces. Schottky surfaces are certain hyperbolic surfaces, hence locally 
symmetric spaces of rank one. Therefore the Weyl chamber flow on Schottky 
surfaces coincides with the geodesic flow, and is a well-studied object. The 
classical Koebe--Morse method gives rise to cross sections and codings for 
the geodesic flow on Schottky surfaces, which have already been used for many 
different purposes. Also we will take advantage of these results for 
our constructions in Sections~\ref{sec:cs} and~\ref{sec:coding}. In this 
section, we briefly present these classical results, with an emphasis on the 
dynamical aspects. We refer to~\cite{Borthwick_book} for details and proofs.

The Riemannian symmetric space we consider in this section is the hyperbolic 
plane. We will use throughout the upper half-plane model 
\[
 \h \coloneqq \h^2 \coloneqq \{ z\in\C \mid \Ima z > 0\}\,,\quad ds^2_z 
\coloneqq \frac{dz\,d\overline z}{(\Ima z)^2}\,,
\]
where $\Ima z$ denotes the imaginary part of $z\in\C$. We identify the identity 
component~$G$ of the group of Riemannian isometries of~$\h$ with the Lie group
$\PSL_2(\R) = \SL_2(\R)/\{\pm I\}$, where $I$ denotes the identity matrix 
in~$\SL_2(\R)$. We denote an element of~$G=\PSL_2(\R)$ by 
\[
 \bmat{a}{b}{c}{d}
\]
if it is represented by the matrix $\textmat{a}{b}{c}{d}\in\SL_2(\R)$. With 
respect to this identification, $G$ acts on~$\h$ by fractional linear 
transformations. Thus, 
\[
 g(z) = \frac{az+b}{cz+d}
\]
for all $g=\textbmat{a}{b}{c}{d}\in G$, $z\in\h$. As origin of~$\h$ we 
pick~$i$, as reference flat the imaginary axis $i\R_{>0}$ and as reference Weyl 
chamber the upper half of the reference flat, thus $i(1,\infty)$. The 
stabilizer group of~$i$ is 
\[
 K = \PSO(2) = \SO(2)/\{\pm I\}\,,
\]
the maximal abelian subgroup~$A$ of~$G$ is
\[
 A = \left\{ a_t \coloneqq \bmat{e^{t/2}}{0}{0}{e^{-t/2}} \ \left\vert\  
t\in\R \vphantom{\bmat{e^{t/2}}{0}{0}{e^{-t/2}}}\right.\right\}\,,
\]
and the positive Weyl chamber in~$A$ is 
\[
 A^+ = \{ a_t \mid t>0\}\,.
\]
The centralizer of~$A$ in~$K$ is the trivial group~$M = \{ \id \}$, and 
the normalizer of~$A$ in~$K$ is 
\[
 M' = \left\{ \id, \bmat{0}{1}{-1}{0} \right\}\,.
\]

We may identify the reference Weyl chamber with the unit tangent vector at~$i$ 
that is tangent to~$i[1,\infty)$ (i.e., the vector at~$i$ that points upwards).
The unipotent subgroup is 
\[
 N = \left\{ \bmat{1}{x}{0}{1} \ \left\vert\  x\in\R 
\vphantom{\bmat{1}{x}{0}{1}}  \right.\right\}\,.
\]
Hence, the associated minimal parabolic subgroup~$P=NAM$ is the 
subgroup of~$G$ that is represented by upper triangular matrices 
in~$\SL_2(\R)$. The Furstenberg boundary~$G/P$ coincides with the geodesic 
boundary of~$\h$, and the action of $G$ on $\h=G/K$ extends continuously 
to~$G/P$. We may identify the Furstenberg boundary~$G/P$ with $P^1_\R = 
\R\cup\{\infty\}$ by means of the isomorphism 
\[
 G/P \to P^1_\R\,,\quad gP \mapsto g(\infty)\,,
\]
where 
\[
 g(\infty) = \frac{a}{c}
\]
for $g=\textbmat{a}{b}{c}{d}\in G$, using the convention $a/0 = \infty$.

A \emph{Schottky surface} is a hyperbolic surface of infinite area without 
cusps and conical singularities. Each Schottky surface (and only those) arises 
from the following construction. We pick $q\in\N$ and fix $2q$ closed Euclidean 
disks in~$\C$ that are centered in~$\R$ and that are pairwise disjoint. We fix 
a pairing of these disks, that we shall indicate by indices with opposite 
signs. Let
\[
 \mc D_1, \mc D_{-1},\,\ldots\,, \mc D_q, \mc D_{-q}
\]
be the chosen disks. For $k\in\{1,\ldots, q\}$ we pick an element~$g_k\in G$ 
that maps the exterior of the disk~$\mc D_k$ to the interior of the 
disk~$\mc D_{-k}$. The subgroup~$\Gamma$ of~$G$ generated by the 
elements~$g_1,\ldots, g_q$, 
\[
 \Gamma = \langle g_1,\ldots, g_q\rangle\,,
\]
is a \emph{Schottky group}, and the hyperbolic surface $\locspace \coloneqq 
\Gamma\backslash\h$ is a \emph{Schottky surface}. The complement of the 
union of the disks in~$\h$, 
\[
 \mc F \coloneqq \h\setminus \bigcup_{k=1}^q \bigl( \mc D_{k} \cup \mc 
D_{-k}\bigr)\,,
\]
is an open fundamental domain for~$\locspace$, all of whose sides in~$\h$ are 
given by geodesics (namely the part of the boundary of the disks that is 
in~$\h$). Its side-pairings are given by the elements~$g_1,\ldots, g_k$. 

In order to provide a cross section for the geodesic flow on~$\locspace$, we 
set $\mc I \coloneqq \{\pm 1,\ldots, \pm q\}$ and, for any $k\in\mc I$, let
$s_k$ be the boundary of~$\mc D_k$ in~$\h$, and $I_k$ denote the part of the 
Furstenberg boundary~$G/P = \R\cup\{\infty\}$ that is exterior to~$\mc D_k$. We 
refer to~$I_k$ as \emph{forward interval}. We recall the map~$\alpha$ 
from~\eqref{eq:combmap} and set 
\[
 C_k \coloneqq \alpha^{-1}(s_k\times I_k)\,.
\]
The set~$C_k$ may be identified with the set of unit tangent vectors of~$\h$ 
that are based on~$s_k$ and ``point into'' the fundamental domain~$\mc F$. We 
set 
\[
 C \coloneqq \bigcup_{k\in \mc I} C_k\qquad\text{and}\qquad \wh C\coloneqq 
\pi^\Gamma(C)\,.
\]
The following statement is immediately implied from the definitions of~$\wh C$ 
and~$C$ and the property of~$\mc F$ to be geodesically convex. It is also 
an immediate consequence of the Koebe--Morse method, which discusses instead 
the neighboring $\Gamma$-translates of~$\mc F$. 

\begin{prop}\label{prop:cross_schottky}
The set~$\wh C$ is a cross section for the geodesic flow on the Schottky 
surface~$\Gamma\backslash\h$, and $C$ is a set of representatives for~$\wh C$.  
\end{prop}

For the presentation of the induced discrete dynamical system, we restrict the 
consideration to the strict cross section contained to~$\wh C$. To that end we 
denote by~$L$ the \emph{limit set} of~$\Gamma$, that is the set of limit 
points of the $\Gamma$-orbit~$\Gamma(z)$ in the Furstenberg boundary~$P^1_\R$, 
where $z$ is any point of~$\h$. We let
\begin{equation}\label{eq:crossstrongrep_schottky}
 C_{\st} \coloneqq \{ c\in C \mid \nu(c) \in L \}
\end{equation}
be the subset of Weyl chambers (or unit tangent vectors) in~$C$ that project 
to the limit set~$L$, and we set 
\begin{equation}\label{eq:crossstrong_schottky}
 \wh C_\st \coloneqq \pi^\Gamma(C_\st)\,.
\end{equation}
(For convenience we allow ourselves here this slight abuse of notation: we have 
not yet shown that $\wh C_\st$ coincides with the set defined 
in~\eqref{eq:crossstrong}. This will be done in 
Proposition~\ref{prop:cross_schottky}.)
For $k\in\mc I$ we set 
\[
 I_{\st,k}^c \coloneqq L \cap (P^1_\R\setminus I_k)\,,
\]
which is the part of the limit set~$L$ contained in the interior of the 
disk~$\mc D_{k}$. We call~$I_{\st,k}^c$ a \emph{strong coding set}, a 
wording whose meaning will become clear in Section~\ref{sec:coding}. Then $L$ 
is the disjoint union of these sets:
\[
 L = \bigcup_{k\in\mc I} I_{\st,k}^c\,.
\]
We define a self-map $F\colon L\to L$ by 
\[
 F\vert_{I_{\st,k}^c} \colon I_{\st,k}^c \to L\,,\quad x\mapsto g_{k}(x)\,,
\]
for $k\in\mc I$. 

As above in Proposition~\ref{prop:cross_schottky}, the following 
statements follow immediately from the definitions of~$\wh C_\st$, $C_\st$ 
and~$F$ as well as the properties of limit sets of Schottky groups. For the 
convenience of the reader, we provide a sketch of the proof.

\begin{prop}\label{prop:coding_schottky}
The set~$\wh C_\st$ is a cross section for the geodesic flow on the Schottky 
surface~$\Gamma\backslash\h$. It is intersected by every geodesic 
on~$\Gamma\backslash\h$ that is contained in a compact subset 
of~$\Gamma\backslash\h$, which may depend\footnote{Due to the special 
structure of Schottky surfaces, we may choose a uniform compact set, namely the 
compact core of~$\Gamma\backslash\h$.} on the considered geodesic. It is the 
strong cross section contained in~$\wh C$. The set~$C_\st$ is a set of 
representatives for~$\wh C_\st$, and the map $F$ is the discrete dynamical 
system on the Furstenberg boundary that is induced by~$C_\st$.
\end{prop}

\begin{proof}[Sketch of proof]
We recall that oriented flats on~$\Gamma\backslash\h$ are precisely the 
geodesics on~$\Gamma\backslash\h$. Let $\wh\gamma$ be such a geodesic and 
suppose that~$\gamma$ is one of its representative geodesics on~$\h$. If 
$\gamma$ is directed towards a point in the limit set~$L$, then while moving 
towards this points, $\gamma$ intersects infinitely many $\Gamma$-translates 
of the fundamental domain~$\mc F$. In turn, $\wh\gamma$ stays ``far away'' from 
the ends of~$\Gamma\backslash\h$ and intersects~$\wh C$ in an unbounded set of 
times. However, if $\gamma$ is directed towards a point \emph{not} in~$L$, then 
eventually $\gamma$ will stay in a single $\Gamma$-translate of~$\mc F$. In 
turn, $\wh\gamma$ will travel to an end of~$\Gamma\backslash\h$, and will  
intersect~$\wh C$ only finitely many times in this direction. From this 
dichotomy, one can easily deduce that $\wh C_\st$ as defined 
in~\eqref{eq:crossstrong_schottky} coincides with the set defined 
in~\eqref{eq:crossstrong}. We remark that this dichotomy takes advantage of 
properties of limit sets that are rather specific to Schottky groups. 

To show that $\wh C_\st$ is indeed a cross section, we first note that 
the compact oriented flats on~$\Gamma\backslash\h$ are precisely the periodic 
geodesics of~$\Gamma\backslash\h$. Let $\wh \gamma=\Gamma gAM$ be such a 
periodic geodesic. By Proposition~\ref{prop:cross_schottky}, $\gamma$ 
intersects~$\wh C$ at least once, say in~$\wh v$. Since all intersections 
of~$\wh\gamma$ and~$\wh C$ are discrete by 
Proposition~\ref{prop:cross_schottky}, and the group~$A$ is one-dimensional, 
every future intersection with~$\wh C$ is a next intersection. Since $\wh\gamma$ 
is periodic, the intersection in~$\wh v$ will repeatedly be among the future 
intersections, showing that these exist unboundedly. Thus, $\wh v$ is in~$\wh 
C_\st$ and $\wh\gamma$ intersects~$\wh C_\st$. Obviously, $C_\st$ is a set of 
representatives of~$\wh C_\st$. 

It remains to indicate why~$F$ is the induced discrete dynamical system. Let 
$\wh v = \Gamma g M \in \wh C_\st$. The geodesic~$\wh\gamma$ determined by~$\wh 
v$ is~$\Gamma g AM$. Without loss of generality, we may assume that $g\in G$ is 
chosen such that $v\coloneqq gM$ is the unique representative of~$\wh v$ 
in~$C_\st$. Then the geodesic~$\gamma$ on~$\h$ determined by~$v$ is~$gAM$, or, 
from a more dynamical point of view, the trajectory 
\[
 t\mapsto ga_tM\,.
\]
Let 
\[
 b_0 \coloneqq \lim_{t\to\infty} ga_tM = g(\infty)
\]
be the point in the Furstenberg boundary of~$\h$ to which $\gamma$ is oriented 
(or projects). Due to the relation between~$C$ and~$\mc F$, the time-minimal 
intersection between
\[
 \gamma^+\coloneqq \{ ga_tM \mid t > 0\}
\]
and the $\Gamma$-translates of~$C$ is located at the (unique) boundary 
component of~$\mc F$ through which~$\gamma^+$ passes. The structure of~$\mc F$ 
implies that this intersection is in~$g_{-k}(C_{-k})$ for some $k\in\mc I$ if 
and only if $b_0\in I_{\st,-k}^c$. Suppose that the intersection is in 
$g_{-k_0}(C_{-k_0})$ at time $t_0>0$, with $k_0\in\mc I$. Then the next 
intersection of~$\wh v$ with~$\wh C_\st$ is in $\Gamma g a_{t_0} M$, which 
corresponds via $\bigl(\pi^\Gamma\vert_C\bigr)^{-1}$ to the element
\[
 g_{-k_0}^{-1}\bigl( ga_{t_0}M \bigr) = g_{k_0}\bigl( ga_{t_0}M\bigr)
\]
of~$C_\st$. In turn,
\begin{align*}
 F(b_0) & = F\bigl( \nu(gM) \bigr) = \nu\bigl( R(gM) \bigr) = \nu\bigl( 
g_{k_0}(ga_{t_0}M) \bigr) = g_{k_0}\nu\bigl( ga_{t_0}M \bigr) = g_{k_0}(b_0)\,.
\end{align*}
This shows that $F$ is indeed the induced discrete dynamical system.
\end{proof}

\section{Cross sections for the Weyl chamber flow on 
product spaces}\label{sec:cs}

Let $r\in\N$. We recall that $\h^2 = \h$ denotes the hyperbolic plane and 
consider the Riemannian symmetric space
\[
 \globspace \coloneqq (\h^2)^r = \h \times \cdots \times \h
\]
of rank~$r$, given by the direct product of $r$ copies of~$\h$. We identify the 
identity component of the group of Riemannian isometries of~$\globspace$ with 
\[
 G \coloneqq \PSL_2(\R)^r\,.
\]
Then the action of~$G$ on~$\globspace$ is
\[
 g(z) = \bigl( g_1(z_1),\ldots, g_r(z_r) \bigr)
\]
for all $g=(g_1,\ldots, g_r)\in G$ and $z=(z_1,\ldots, z_r)\in\globspace$. For 
$j\in\{1,\ldots, r\}$ we choose a (Fuchsian) Schottky group~$\Gamma_j$ 
in~$\PSL_2(\R)$ and set 
\[
 \Gamma \coloneqq \Gamma_1\times \cdots \times \Gamma_r\,.
\]
In this section we construct a cross section for the Weyl chamber flow on the 
locally symmetric space 
\[
 \locspace \coloneqq \Gamma\backslash\globspace\,.
\]
We start with some preparatory considerations. Since $\globspace$ as well as 
\[
 \locspace = \Gamma_1\backslash\h \times \cdots \times \Gamma_r\backslash\h
\]
enjoy clear product structures, several of the necessary objects are the direct 
products of the analogous objects of the single factors. As origin 
of~$\globspace$ we choose
\[
 x_0 \coloneqq (i,\ldots, i) \in \globspace\,.
\]
The flats and Weyl chambers of~$\globspace$ are the direct products of the 
flats and Weyl chambers of~$\h$. Therefore we choose $F_0\coloneqq 
(i\R_{>0})^r$ as \emph{reference flat} of~$\globspace$, and 
$c_0\coloneqq (i(1,\infty))^r$ as \emph{reference Weyl chamber} 
of~$\globspace$, which are the direct products of our chosen reference flat and 
reference Weyl chamber of~$\h$. 

In Section~\ref{sec:schottky} we discussed the groups and maps associated to 
our choices of reference objects of~$\h$. In what follows we will use for these 
groups and maps the notation from Section~\ref{sec:schottky} but with the 
additional subscript ``1''. Thus, $G_u = \PSL_2(\R)$, $K_u= \PSO(2)$, etc. 
The subscript-free notation is preserved for the objects related 
to~$\globspace$ and~$\locspace$. The stabilizer group of~$x_0$ in~$G$ is 
\[
 K = \Stab_G(x_0) = K_u^r = \PSO(2)^r\,,
\]
the maximal abelian subgroup of~$G$ determined by~$F_0$ is~$A = A_u^r$, the 
positive Weyl chamber in~$A$ determined by~$c_0$ is $A^+ = (A_u^+)^r$. The 
centralizer group and normalizer group of~$A$ in~$K$ are $M = M_u^r$ and 
$M'=(M'_u)^r$, respectively, and the Weyl group is $W = W_u^r$. The 
unipotent subgroup is $N = N_u^r$, and the minimal parabolic subgroup is $P = 
P_u^r$. Thus, the Furstenberg boundary of~$\globspace$ is 
\[
 G/P = G_u/P_u \times \cdots \times  G_u/P_u\,,
\]
the $r$-times direct product of the Furstenberg boundary of~$\h$, which we 
identify with $\bigl( P^1_\R \bigr)^r$ via the isomorphism
\[
 gP = (g_1P_u,\ldots, g_rP_u) \mapsto \bigl(g_1(\infty),\ldots, g_r(\infty) 
\bigr)\,.
\]
For $j\in\{1,\ldots, r\}$ we fix a fundamental domain~$\mc F_j$ for the 
Fuchsian Schottky group~$\Gamma_j$ in~$\h$ of the form as in 
Section~\ref{sec:schottky}, arising from the choice of $q_j$ Euclidean disks. 
We set $\mc I_j = \{\pm 1,\ldots, \pm q_j\}$, let 
\begin{equation}\label{eq:notationfactors}
 g_{j,k}\,,\ s_{j,k}\,,\ I_{j,k}\qquad\text{for $k\in\mc I_j$}
\end{equation}
denote the side-pairing elements in~$\Gamma_j$, the geodesic sides of~$\mc 
F_j$, and the forward intervals, respectively. We recall the map~$\alpha$ 
from~\eqref{eq:combmap} and set 
\[
 \mc J \coloneqq \bigtimes_{j=1}^r \mc I_j\,.
\]
For each $m = (m_1,\ldots, m_r)\in \mc J$ we set
\[
 Q_m \coloneqq \bigtimes_{j=1}^r \bigl( s_{j,m_j} \times I_{j,m_j}\bigr)\,,
\qquad
 C_m \coloneqq \alpha^{-1}( Q_m )\,,
\]
and 
\[
 C \coloneqq \bigcup_{m\in\mc J} C_m\qquad\text{and}\qquad \wh C \coloneqq 
\pi^\Gamma(C)\,,
\]
where $\pi^\Gamma$ is the map in~\eqref{eq:projWC}. 

\begin{thm}\label{thm:cs_wcf}
The set~$\wh C$ is a cross section for the Weyl chamber flow on~$\locspace$, 
and $C$ is a set of representatives for~$\wh C$. 
\end{thm}

Preparatory for the proof we briefly discuss the product structure of~$\wh C$ 
and~$C$. To that end, for $j\in\{1,\ldots, r\}$ and $m_j \in \mc I_j$, we set 
 
\[
 C_{j,m_j} \coloneqq \alpha_u^{-1}\bigl( s_{j,m_j} \times I_{j,m_j} \bigr)\,, 
\]
where $\alpha_u$ denotes the map in~\eqref{eq:combmap} for~$G_u$. 
Further, we set 
\[
 C_j \coloneqq \bigcup_{m_j\in\mc I_j} C_{j, m_j}\qquad\text{and}\qquad \wh 
C_j \coloneqq \pi_u^\Gamma\bigl( C_j \bigr)\,.
\]
Then, for any $m = (m_1,\ldots, m_r)\in \mc J$, we have
\[
 C_m = \bigtimes_{j=1}^r C_{j,m_j}\,,
\]
and further
\[
 \wh C = \bigtimes_{j=1}^r \wh C_j\,.
\]
For each $j\in\{1,\ldots, r\}$, the set~$\wh C_j$ is a cross section for the 
Weyl chamber flow (geodesic flow) on the Schottky 
surface~$\Gamma_j\backslash\h$ with $C_j$ as set of representatives by 
Proposition~\ref{prop:cross_schottky}.

\begin{proof}[Proof of Theorem~\ref{thm:cs_wcf}]
Let $\Gamma g AM$ be an oriented compact flat of~$\locspace$. With 
$g=(g_1,\ldots, g_r)$ we have 
\[
 \Gamma g AM = (\Gamma_1 g_1 A_uM_u,\cdots,\Gamma_r g_r A_uM_u)\,.
\]
Thus, for each $j\in\{1,\ldots, r\}$,  
\[
 \gamma_j \coloneqq \Gamma_j g_j A_uM_u 
\]
is an oriented flat (oriented geodesic) on~$\Gamma_j\backslash\h$ that is 
contained in a compact subset of~$\Gamma_j\backslash\h$. It is even a periodic 
geodesic.  Thus, $\gamma_j$ intersects~$\wh C_j$ by 
Proposition~\ref{prop:cross_schottky}, say in~$\Gamma_jg_ja_jM_u$. Then $\Gamma 
g AM$ intersects~$\wh C$ in $\Gamma g (a_1,\ldots, a_r)M$. This 
establishes~\ref{cs1} for~$\wh C$. In order to show~\ref{cs2} let $\Gamma g 
M\in\wh C$ and suppose that $g=(g_1,\ldots, g_r)$. Then 
\[
 \Gamma g M = (\Gamma_1 g_1 M_u,\ldots, \Gamma_r g_r M_u) \in \wh C_1 \times 
\cdots \times \wh C_r\,.
\]
Thus, for each $j\in\{1,\ldots, r\}$, the Weyl chamber (unit tangent 
vector)~$\Gamma_j g_j M_u$ is in the cross section~$\wh C_j$ for the geodesic 
flow on~$\Gamma_j\backslash\h$. Hence we find $\eps_j>0$ such that for all 
$t\in (-\eps_j,\eps_j)$, $t\not=0$, 
\[
 \Gamma_j g_j a_t M_u \notin \wh C_j\,,
\]
where 
\[
 a_t = \bmat{e^{t/2}}{0}{0}{e^{-t/2}} \in G_u = \PSL_2(\R)\,.
\]
For 
\[
 U \coloneqq (-\eps_1,\eps_1)\times \cdots \times (-\eps_r,\eps_r)
\]
we have that each element $a\in U\cap A^\reg$ is of the form 
\[
 a = (a_{t_1},\ldots, a_{t_r})
\]
with $t_j\in (-\eps_j, \eps_j)$, $t_j\not=0$, for $j\in\{1,\ldots, r\}$, and 
hence
\[
 \Gamma ga M = \bigl( \Gamma_1g_1a_{t_1}M_u, \ldots, \Gamma_rg_ra_{t_r}M_u 
\bigr)\notin \wh C\,.
\]
This establishes~\eqref{cs2} for~$\wh C$ and finishes the proof that~$\wh C$ is 
a cross section for the Weyl chamber flow on~$\locspace$. Finally for each 
$j\in\{1,\ldots, r\}$, the set $C_j$ is a set of representatives for~$\wh 
C_j$. This property is stable under direct products, and hence $C$ is indeed 
a set of representatives for~$\wh C$. 
\end{proof}

We briefly discuss an aspect of the structure and thickness of~$\wh C$. For 
each $j\in\{1,\ldots, r\}$ the set of base points of the cross section~$\wh C_j$ 
is the full boundary of the fundamental domain~$\mc F_j$. A fundamental domain 
for~$\locspace = \Gamma\backslash\globspace$ is given by 
\[
 \mc F \coloneqq \bigtimes_{j=1}^r \mc F_j\,.
\]
The set of base points of~$\wh C$, however, is only a rather sparse subset of 
the boundary of~$\mc F$, getting sparser if the rank becomes larger. This shows 
that the cross section~$\wh C$ is not implied by a Koebe--Morse method, but is 
a genuinely related to the Weyl chamber bundle. 

In a way similar to the construction of~$\wh C$ we can find the strong cross 
section contained in~$\wh C$. To that end, for $j\in\{1,\ldots, r\}$, we let 
$L_j$ denote the limit set of~$\Gamma_j$ (in the rank one situation), set 
\[
 C_{\st, j} \coloneqq \{ c\in C_j \mid \nu_1(c) \in L_j\}
\]
as well as 
\[
 \wh C_{\st, j} \coloneqq \pi_u^\Gamma\bigl( C_{\st, j}\bigr)\,.
\]
As stated in Proposition~\ref{prop:coding_schottky}, $\wh C_{\st, j}$ is the 
strong cross section contained in~$\wh C_j$ for the geodesic flow 
on~$\Gamma_j\backslash\h$. We further define 
\[
 C_\st\coloneqq \bigtimes_{j=1}^k C_{\st,j}\qquad\text{and}\qquad\wh C_\st 
\coloneqq \bigtimes_{j=1}^r \wh C_{\st,j}\,.
\]
Then $\wh C_\st = \pi^\Gamma\bigl(C_\st\bigr)$.

\begin{thm}\label{thm:stcross_wcf}
The set~$\wh C_\st$ is the strong cross section contained in~$\wh C$, and 
$C_\st$ is a set of representatives.
\end{thm}

\begin{proof}
These statements can be proven analogously to those in 
Theorem~\ref{thm:cs_wcf}, using Proposition~\ref{prop:coding_schottky} 
instead of Proposition~\ref{prop:cross_schottky}. 
\end{proof}

\section{Codings and discrete dynamical systems on Furstenberg 
boundary}\label{sec:coding}

In this section we present the discrete dynamical system on the Furstenberg 
boundary that is induced by the set of representatives~$C_\st$ for the strong 
cross section~$\wh C_\st$ from Section~\ref{sec:cs}. We resume the notation 
from Section~\ref{sec:cs}. 

For $j\in\{1,\ldots, r\}$, we denote the strong coding sets determined by the 
choice of the fundamental domain~$\mc F_j$ for $\Gamma_j$ by 
\[
 I_{\st,j,k}^c\qquad\text{for $k\in\mc I_j$}.
\]
For $m\in \mc J = \bigtimes_{j=1}^k \mc I_j$, $m=(m_1,\ldots, m_r)$, we set
\[
 I_{\st, m}^c \coloneqq \bigtimes_{j=1}^r I_{\st, j, m_j}^c
\]
and 
\[
 D \coloneqq \bigcup_{m\in\mc J} I_{\st,m}^c\,.
\]
We define the map $F\colon D\to D$ as follows: for $m\in \mc J$, 
$m=(m_1,\ldots, m_r)$, we set 
\[
 g_m \coloneqq \bigl(g_{1,m_1},\ldots, g_{r,m_r}\bigr)\,,
\]
where $g_{j,m_j}$ are the side-pairing elements of~$\mc F_j$ 
(see~\eqref{eq:notationfactors}). Restricted to the subset~$I_{\st,m}^c$ 
of~$D$, the map~$F$ is 
\begin{equation}\label{eq:inducedmap}
 F\vert_{I_{\st,m}^c}\colon I_{\st,m}^c\to D\,,\quad x\mapsto g_m(x)\,.
\end{equation}
Further, for $j\in\{1,\ldots, r\}$ let 
\begin{equation}\label{eq:Fj}
F_j\colon L_j\to L_j
\end{equation}
denote the discrete dynamical system induced on the 
Furstenberg boundary~$P^1_\R$ of~$\h$ induced by the cross section~$\wh 
C_{\st,j}$ and its set of representatives~$C_{\st,j}$ for the geodesic flow 
on~$\Gamma_j\backslash\h$.

\begin{thm}\label{thm:inducedmap}
\begin{enumerate}[label=$\mathrm{(\roman*)}$, ref=$\mathrm{\roman*}$]
\item\label{mapi} The map~$F$ is the discrete dynamical system on 
Furstenberg boundary that is induced by~$C_{\st}$. 
\item\label{mapii} For any $m=(m_1,\ldots, m_r)\in\mc J$ we have 
\[
 F\vert_{I_{\st, m}^c} = \left( F_1\vert_{I_{\st,1,m_1}^c},\ldots, 
F_r\vert_{I_{\st,r,m_r}^c} \right)\,.
\]
\end{enumerate}
\end{thm}

\begin{proof}
The statement in~\eqref{mapii} follows immediately from the definition of the 
map~$F$. To establish~\eqref{mapi}, we recall that the $r$-dimensional Weyl 
chamber flow on~$\locspace$ is the direct product of the $1$-dimensional 
geodesic flows on the factors~$\Gamma_j\backslash\h$, $j\in\{1,\ldots,r\}$, 
of~$\locspace$. We recall further that the cross section~$\wh C_\st$ and its 
set of representatives~$C_\st$ are direct products of the cross sections and 
sets of representatives for these factors. Therefore, this product structure 
descends to the induced discrete dynamical systems on Furstenberg boundary, 
which yields~\eqref{mapi} due to the identity in~\eqref{mapii}.
\end{proof}

\section{Transfer operators}\label{sec:TO}

In this final section we propose a definition of a transfer operator 
family for the multi-dimensional discrete dynamical system~$F$ in 
Section~\ref{sec:coding}. We continue to use the notation from 
Sections~\ref{sec:cs} and~\ref{sec:coding} and start by presenting the 
well-known definition of transfer operator families for one-dimensional flows, 
specialized to our setup. 

Let $j\in\{1,\ldots, r\}$ and recall from~\eqref{eq:Fj} and 
Section~\ref{sec:schottky} the discrete dynamical system~$F_j\colon L_j\to L_j$ 
induced by the set of representatives~$C_j$ of the cross section~$\wh C_j$ for 
the geodesic flow on the Schottky surface~$\Gamma_j\backslash\h$. The 
Ruelle-type \emph{transfer operator}~$\TO_{j,s}$ with parameter~$s\in\C$ 
associated to~$F_j$ is (at least initially) an operator on the space of 
functions~$\psi\colon L_j\to \C$, given by 
\begin{equation}\label{eq:initialdefTO}
 \TO_{j,s}\psi(x) \coloneqq \sum_{y\in F_j^{-1}(x)} | F'_j(y)|^{-s} \psi(y) 
\qquad (x\in L_j)\,.
\end{equation}
Here, the derivative of~$F'_j$ at~$y$ is understood as follows. The point~$y$ 
is contained in~$I_{\st,j,k}^c$ for a unique~$k\in\mc I_j$. Then~$F_j$ acts 
on a small neighborhood of~$y$ in~$L_j$ by the fractional linear 
transformation~$g_{j,k}$, which extends to an analytic map in a small 
neighborhood of~$y$ in~$\R$. We use the derivative of this extended map 
for~$F'_j(y)$.  

The function space which one uses as domain for the transfer 
operator~$\TO_{j,s}$ depends on its further applications. One may choose 
spaces of functions with larger domain or with some regularity properties. We 
will refrain here from these discussions and will use the space of functions on 
the limit sets as place holder.

We shall now provide another presentation of~$\TO_{j,s}$ that takes advantage 
of the explicit description of~$F_j$. To that end we note that $F_j$ restricts 
to the bijections 
\[
 I_{\st,j,k}^c \to L_j\setminus I_{\st,j,-k}^c\,,\quad x\mapsto 
g_{j,k}(x)\,,
\]
for each $k\in\mc I_j$. Thus, for each $k\in \mc I_j$, each $x\in I_{\st,j, 
k}^c$ has the ($|\mc I_j|-1$) preimages 
\[
\{g_{j,\ell}^{-1}(x) \mid\ell\in\mc I_j, \ell\not=-k\}\,.
\]
For any function $\psi\colon L_j\to\C$, we set 
\[
 \psi_{k} \coloneqq \psi\cdot 1_{I_{\st,j,k}^c}\qquad\text{for $k\in\mc 
I_j$}\,,
\]
where $1_A$ denotes the characteristic function of the set~$A$. Then 
\[
 \psi = \sum_{k\in\mc I_j} \psi_k\,.
\]
Further, for $h\in\Gamma_j$, $s\in\C$, any subset $A\subseteq\R$ and any 
function $\varphi\colon A\to\C$ we set 
\begin{equation}\label{eq:tau_one}
 \tau_s(h^{-1})\varphi(x) \coloneqq \bigl(h'(x)\bigr)^s \varphi\bigl( h(x) 
\bigr)\qquad (x\in A)\,,
\end{equation}
whenever it is well-defined (as it will be in all our applications). For each 
$k\in\mc I_j$ we have then
\[
 \bigl(\TO_{j,s}\psi\bigr)_k = \sum_{\substack{\ell\in \mc I_j \\ \ell\not=-k}} 
\tau_s\bigl(g_{j,\ell}\bigr)\psi_\ell
\]
or, in a more compact form,
\[
 \TO_{j,s} = \sum_{k\in\mc I_j} 1_{I_{\st,j,k}^c} \cdot \sum_{\substack{\ell\in 
\mc I_j \\ \ell\not=-k}} \tau_s\bigl(g_{j,\ell}\bigr)\,.
\]

For the transfer operator family associated to the multi-dimensional map~$F$ we 
propose a definition analogous to those in~\eqref{eq:initialdefTO} but allowing 
a multi-dimensional parameter $s\in \C^r$. (We recall that the rank of the 
considered Riemannian locally symmetric space~$\locspace$ is~$r$.) We first 
consider the parameter-free transfer operator 
\begin{equation}\label{eq:TO_wo_param}
 \TO f(x) \coloneqq \sum_{y\in F^{-1}(x)} |F'(y)|^{-1} f(y)\,,
\end{equation}
acting on functions~$f\colon D\to\C$, where the derivative $F'$ is understood 
analogously to above and $|F'(y)|$ is the absolute value of the determinant of 
the linear map $F'(y)$. For any $y\in D$ we find a unique element $m\in\mc J$, 
$m=(m_1,\ldots, m_r)$, such that $y\in I_{\st,m}^c$. Letting 
\[
 y = (y_1,\ldots, y_r)\qquad\text{and}\qquad g_m = \bigl( g_{1,m_1},\ldots, 
g_{r,m_r} \bigr)
\]
we then have $F=g_m$ is a small neighborhood of~$y$ in~$D$. Thus, 
\[
 F(y) = g_m(y) = \bigl( g_{1,m_1}(y_1),\ldots, g_{r,m_r}(y_r) \bigr)
\]
and the Jacobi matrix of~$F$ at~$y$ is the diagonal matrix
\begin{equation}\label{eq:JM}
 J_F(y) =
 \begin{pmatrix}
 g_{1,m_1}'(y_1) 
 \\
 & \ddots 
 \\
 & & g_{r,m_r}'(y_r)
 \end{pmatrix}\,.
\end{equation}
Thus, 
\begin{equation}\label{eq:det_unweight}
 |F'(y)|^{-1} = |\det J_F(y)|^{-1} = \prod_{j=1}^r 
\bigl(g_{j,m_j}'(y_j)\bigr)^{-1}\,.
\end{equation}
Motivated by the diagonal structure of the Jacobi matrix in~\eqref{eq:JM}, we 
propose to endow each non-zero entry separately with a weight. Thus, the 
\emph{parametrized transfer operator}~$\TO_s$ with $s=(s_1,\ldots, s_r)\in 
\C^r$ is defined as
\begin{equation}\label{eq:TO_weighted}
 \TO_sf(x) = \sum_{y\in F^{-1}(x)} |F'(y)|^{-s} f(y)\,,
\end{equation}
where 
\[
 |F'(y)|^{-s} \coloneqq \prod_{j=1}^r \bigl(g_{j,m_j}'(y_j)\bigr)^{-s_j}
\]
in the notation from~\eqref{eq:det_unweight}. This use of the parameter also 
reflects well the independence of the $r$ dimensions of the oriented flats 
of~$\locspace$. 

The analogy between the transfer operators in~\eqref{eq:TO_weighted} 
and~\eqref{eq:initialdefTO} goes further. For $m=(m_1,\ldots, m_r)\in\mc J$ we 
set
\[
B(m) \coloneqq \{ n=(n_1,\ldots, n_r)\in\mc J \mid \exists\, j\in\{1,\ldots, 
r\}\colon n_j = -m_j\}\,.
\]
The map $F\colon D\to D$ restricts to the bijections 
\[
 I_{\st,m}^c \to D\setminus\!\!\bigcup_{n\in B(m)}I_{\st,n}^c\,,\quad x\mapsto 
g_m(x)\,.
\]
In analogy to~\eqref{eq:tau_one}, we define for $h=(h_1,\ldots, h_n)\in\Gamma$, 
$s=(s_1,\ldots, s_r)\in\C^r$, any subset $A\subseteq\R^r$ and any function 
$\varphi\colon A\to\C$, 
\begin{align}
 \omega_s(h^{-1})\varphi(x) & = \bigl| h'(y) \bigr|^{-s} \varphi\bigl(h(y)\bigr)
 \\
 & = \bigl|h'_1(y_1)\bigr|^{-s_1} \cdots \bigl|h'_r(y_r)\bigr|^{-s_r} 
\varphi\bigl(h(y)\bigr)\,. \nonumber
\end{align}
Then 
\begin{equation}
 \TO_s = \sum_{m\in\mc J} 1_{I_{\st,m}^c} \cdot \sum_{\substack{n\in\mc J\\ 
n\notin B(m)}} \omega_s(g_n)\,,
\end{equation}
or, if we set 
\[
 f_m \coloneqq f\cdot 1_{I_{\st,m}^c} \qquad (m\in\mc J)
\]
for any function $f\colon D\to\C$, then 
\begin{equation}
 (\TO_sf)_m = \sum_{\substack{n\in\mc J\\ n\notin B(m)}} \omega_s(g_n)f_n\,.
\end{equation}
We remark that this multi-parameter transfer operator for the Weyl chamber flow 
on~$\locspace$ is \emph{not} the direct sum of the one-parameter transfer 
operators for the geodesic flow on the Schottky surfaces from which $\locspace$ 
is build. This is consistent with the fact that the spectral theory 
of~$\locspace$ is not just the ``direct product'' of the spectral theories of 
its factors. We leave any further investigation into this transfer operator 
family for future work.

\bibliography{APbib}
\bibliographystyle{amsplain}

\setlength{\parindent}{0pt}

\end{document}